\documentclass{amsart}
\usepackage{graphicx}
\usepackage{hyperref}
\usepackage{amssymb}
\usepackage{amsthm}
\usepackage{amsmath}

\newtheorem{theorem}{Theorem}[section]
\newtheorem{lemma}[theorem]{Lemma}

\newtheorem{proposition}[theorem]{Proposition}
\newtheorem{conjecture}[theorem]{Conjecture}
\theoremstyle{definition}

\theoremstyle{remark}
\newtheorem{remark}[theorem]{Remark}

\numberwithin{equation}{section}

\def\limfunc#1{\mathop{\rm #1}\nolimits}

\begin{document}

\title[Mirror coupling of RBM and applications]{Mirror coupling of reflecting Brownian motion and an application to
Chavel's conjecture}
\author{Mihai N. Pascu}
\address{Faculty of Mathematics and Computer Science, Transilvania University of Bra\c{s}ov, Bra\c{s}ov -- 500091, ROMANIA}
\email{mihai.pascu@unitbv.ro}
\thanks{The author kindly acknowledges the support from CNCSIS-UEFISCSU research grant PNII - IDEI 209.}

\maketitle

\begin{abstract}
In a series of papers, Burdzy et. al. introduced the \emph{mirror
coupling} of reflecting Brownian motions in a smooth bounded domain
$D\subset \mathbb{R}^{d}$, and used it to prove certain properties
of eigenvalues and eigenfunctions of the Neumann Laplaceian on $D$.

In the present paper we show that the construction of the mirror
coupling can be extended to the case when the two Brownian motions
live in different domains $D_{1},D_{2}\subset \mathbb{R}^{d}$.

As an application of the construction, we derive a unifying proof of
the two main results concerning the validity of Chavel's conjecture
on the domain monotonicity of the Neumann heat kernel, due to I.
Chavel (\cite{Chavel}), respectively W. S. Kendall (\cite{Kendall}).
\end{abstract}

\section{Introduction}

The technique of coupling of reflecting Brownian motions is a useful
tool used by several authors in connection to the study of the
Neumann heat kernel of the corresponding domain (see
\cite{AtarBurdzy2}, \cite{Banuelos - On the Hot Spots},
\cite{Burdzy}, \cite{CarmonaZheng}, \cite{Kendall},
\cite{Pascu-scaling}, etc).

In a series of paper, Krzysztof Burdzy et. al. ( \cite{AtarBurdzy1},
\cite{AtarBurdzy2}, \cite{Banuelos - On the Hot Spots},
\cite{Burdzy}, \cite{Burdzy-Kendall},) introduced the \emph{mirror
coupling} of reflecting Brownian motions in a smooth domain
$D\subset \mathbb{R}^{d}$ and used it in order to derive properties
of eigenvalues and eigenfunctions of the Neumann Laplaceian on $D$.

In the present paper, we show that the mirror coupling can be extended to
the case when the two reflecting Brownian motions live in different domains $%
D_{1},D_{2}\subset \mathbb{R}^{d}$.

The main difficulty in the extending the construction of the mirror
coupling comes from the fact that the stochastic differential
equation(s) describing the mirror coupling has a singularity at the
times when coupling occurs. In the case $D_{1}=D_{2}=D$ considered
by Burdzy et. all. this problem is not a major
problem (although the technical details are quite involved, see \cite%
{AtarBurdzy2}), since after the coupling time the processes move
together. In the case $D_{1}\neq D_{2}$ however, this is a major
problem: after processes have coupled, it is possible for them to
decouple (for example in the case when the processes are coupled and
they hit the boundary of one of the domains).

It is worth mentioning that the method used for proving the existence of the
solution is new, and it relies on the additional hypothesis that the smaller
domain $D_{2}$ (or more generally $D_{1}\cap D_{2}$) is a convex domain.
This hypothesis allows us to construct an explicit set of solutions in a
sequence of approximating polygonal domains for $D_{2}$, which converge to
the desired solution.

As an application of the construction, we will derive a unifying
proof of the two most important results on the challenging Chavel's
conjecture on the domain monotonicity Neumann heat kernel
(\cite{Chavel}, \cite{Kendall}), which also gives a possible new
line of approach for this conjecture (note that by the results in
\cite{Bass-Burdzy}, Chavel's conjecture does not hold in its full
generality, but the additional hypotheses under which this
conjecture holds are not known at the present moment).

The structure of the paper is as follows: in Section \ref{Mirror
coupling of RBM} we briefly describe the construction of Burdzy et.
al. of the mirror coupling in a smooth bounded domain domain
$D\subset \mathbb{R}^{d}$.

In Section \ref{Extensions of mirror coupling}, in Theorem \ref{main
theorem}, we give the main result which shows that the mirror
coupling can be extended to the case when $\overline{D_{2}}\subset
D_{1}$ are smooth bounded domains in $\mathbb{R}^{d}$ and $D_{2}$ is
convex (some extensions of the theorem are presented in Section
\ref{Extensions and applications}).

Before proceeding with the proof of theorem, in Remark
\ref{reduction to the case D_1=R} we show that the proof of the
theorem can be reduced to the case when $D_{1}=\mathbb{R}^{d}$.
Next, in Section \ref{1-dimensional case}, we show that in the case
$D_{2}=\left( 0,\infty \right) \subset D_{1}=R$ the solution is
essentially given by Tanaka's formula (Remark \ref{remark on
Tanaka formula}), and then we give the proof of the main theorem in the $1$%
-dimensional case (Proposition \ref{Prop 1-dim case}).

In Section \ref{Case of polygonal domains}, we first prove the existence of
the mirror coupling in the case when $D_{2}$ is a half-space in $\mathbb{R}%
^{d}$ and $D_{1}=\mathbb{R}^{d}$ (Lemma \ref{lemma for a
half-space}), and then we use this result in order to prove the
existence of the mirror coupling in the case when $D_{2}$ is a
convex polygonal domain in $\mathbb{R}^{d}$ and $
D_{1}=\mathbb{R}^{d}$ (Theorem \ref{coupling in polygonal domains}).
Some of the properties of coupling, essential for the extension to
the general case are detailed in Proposition \ref{properties of
mirror coupling in polygonal domains}.

In Section \ref{The proof of the main theorem} we give the proof of
the main Theorem \ref{main theorem}. The idea of the proof is to
construct a sequence $\left( Y_{t}^{n},X_{t}\right) $ of mirror
couplings in $\left( D_{n},D_{1}\right) $, where $D_{n}\nearrow
D_{2}$ is a sequence of convex polygonal domains in $\mathbb{R}^d$,
and to use the properties of the mirror couplings in polygonal
domains (Proposition \ref{properties of mirror coupling in polygonal
domains}) in order to show that the sequence $Y_{t}^{n}$ converges
to a process $Y_{t}$, which gives the desired solution to the
problem.

The last section of the paper (Section \ref{Extensions and applications}) is
devoted to discussing the applications and the extensions of the mirror
coupling constructed in Theorem \ref{main theorem}. First, in Theorem \ref%
{Theorem on Chavel's conjecture} we use the mirror coupling in order to give
a simple, unifying proof of the results of I. Chavel and W. S. Kendall on
the domain monotonicity of the Neumann heat kernel (Chavel's Conjecture \ref%
{Chavel's conjecture}). The proof is probabilistic in spirit, relying on the
geometric properties of the mirror coupling.

In Remark \ref{possible extensions of Chavel}, we discuss the
equivalent analytic counterpart of the proof Theorem \ref{Theorem on
Chavel's conjecture}, which might give a possible new line of
approach for extending Chavel's conjecture to other classes of
domains.

Without giving all the technical details, we discuss the extension of the
mirror coupling to other classes of domains (smooth bounded domains $%
D_{1,2}\subset \mathbb{R}^{d}$ with non-tangential boundaries, such that $%
D_{1}\cap D_{2}$ is a convex domain).

The paper concludes with a discussion on the (non) uniqueness of the mirror
coupling. It is shown here that the lack of uniqueness is due to the fact
that after coupling, the processes might decouple, not only on the boundary
of the domain, but even when they are inside of it.

The two basic solutions give rise to the \emph{sticky}, respectively \emph{%
non-sticky} mirror couplings, and there is a whole range of
intermediate possibilities. The stickiness refers to the fact that
after coupling the processes \textquotedblleft
stick\textquotedblright\ to each other as long as possible (this is
the coupling constructed in Theorem \ref{main theorem}) hence the
name \textquotedblleft sticky\textquotedblright\ mirror coupling, or
they can split apart immediately after coupling, in the case of
\textquotedblleft non-sticky\textquotedblright\ mirror coupling, the
general case (\emph{weak/mild} mirror couplings) being a mixture of
these two basic behaviors.

We developed the extension of the mirror coupling having in mind the
application to Chavel's conjecture, for which the sticky mirror
coupling is the \textquotedblleft right\textquotedblright\ tool, but
perhaps the other mirror couplings (the non-sticky and the mild
mirror couplings) might prove useful in other applications.

\section{Mirror couplings of reflecting Brownian motions\label{Mirror
coupling of RBM}}

Reflecting Brownian motion in a smooth domain $D\subset \mathbb{R}^{d}$ can
be defined as a solution of the stochastic differential equation%
\begin{equation}
X_{t}=x+B_{t}+\int_{0}^{t}\nu _{D}\left( X_{s}\right) dL_{s}^{X},
\end{equation}%
where $B_{t}$ is a $d$-dimensional Brownian motion, $\nu _{D}$ is the inward
unit normal vector field on $\partial D$ and $L_{t}^{X}$ is the boundary
local time of $X_{t}$ (the continuous non-decreasing process which increases
only when $X_{t}\in \partial D$).

In \cite{AtarBurdzy1}, the authors introduced the \emph{mirror coupling} of
reflecting Brownian motion in a smooth domain $D\subset \mathbb{R}^{d}$
(piecewise $C^{2}$ domain in $\mathbb{R}^{2}$ with a finite number of convex
corners or a $C^{2}$ domain in $\mathbb{R}^{d}$, $d\geq 3$).

They considered the following system of stochastic differential equations:%
\begin{eqnarray}
X_{t} &=&x+W_{t}+\int_{0}^{t}\nu _{D}\left( X_{s}\right) dL_{s}^{X} \\
Y_{t} &=&y+Z_{t}+\int_{0}^{t}\nu _{D}\left( X_{s}\right) dL_{s}^{Y} \\
Z_{t} &=&W_{t}-2\int_{0}^{t}\frac{Y_{s}-X_{s}}{\left\vert \left\vert
Y_{s}-X_{s}\right\vert \right\vert ^{2}}\left( Y_{s}-X_{s}\right) \cdot
dW_{s}  \label{Z-W relation}
\end{eqnarray}%
for $t<\xi $, where $\xi =\inf \left\{ s>0:X_{s}=Y_{s}\right\} $ is the
coupling time of the processes, after which the processes $X$ and $Y$ evolve
together, i.e. $X_{t}=Y_{t}$ for $t\geq \xi $ and $Z_{t}=Z_{\xi }+1_{t\geq
\xi }\left( W_{t}-W_{\xi }\right) $.

In the notation of \cite{AtarBurdzy1}, considering the Skorokhod map $\Gamma
:C\left( [0,\infty ):\mathbb{R}^{d}\right) \rightarrow C\left( [0,\infty ):%
\bar{D}\right) $, we have $X=\Gamma \left( x+W\right) $ and $Y=\Gamma \left(
y+Z\right) $, and the above system reduces to
\begin{equation}
Z_{t}=\int_{0}^{t\wedge \xi }G\left( \Gamma \left( y+Z\right) _{s}-\Gamma
\left( x+W\right) _{s}\right) dW_{s}+1_{t\geq \xi }\left( W_{t}-W_{\xi
}\right) ,  \label{Diff eq for Z}
\end{equation}%
where $\xi =\inf \left\{ t\geq 0:\Gamma \left( x+W_{s}\right) =\Gamma \left(
y+Z_{s}\right) \right\} $, for which the authors proved the pathwise
uniqueness and the strong uniqueness of the process $Z_{t}$ (given the
Brownian motion $W_{t}$).

In the above $G:\mathbb{R}^{d}\rightarrow \mathcal{M}_{d\times d}$ denotes
the function defined by
\begin{equation}
G\left( z\right) =\left\{
\begin{tabular}{ll}
$H\left( \frac{z}{\left\vert z\right\vert }\right) ,\qquad $ & if $z\neq 0$
\\
$0,$ & if $z=0$%
\end{tabular}%
\right.  \label{def of G}
\end{equation}%
where for a unitary vector $m\in \mathbb{R}^{d}$, $H\left( m\right) $
represents the linear transformation given by the $d\times d$ matrix
\begin{equation}
H\left( m\right) =I-2m~m^{\prime },
\end{equation}%
that is%
\begin{equation}
H\left( m\right) v=v-2\left( m\cdot v\right) m  \label{def of H}
\end{equation}%
is the mirror image of $v\in \mathbb{R}^{d}$ with respect to the hyperplane
through the origin perpendicular to $m$ ($m^{\prime }$ denotes the transpose
of the vector $m$, vectors being considered as column vectors).

The pair $\left( X_{t},Y_{t}\right) _{t\geq 0}$ constructed above is called
a \emph{mirror coupling }of reflecting Brownian motions in $D$ starting at $%
x,y\in \bar{D}$.

\begin{remark}
The relation (\ref{Z-W relation}) can be written%
\begin{equation*}
dZ_{t}=G\left( \frac{X_{t}-Y_{t}}{\left\vert \left\vert
X_{t}-Y_{t}\right\vert \right\vert }\right) dW_{t},
\end{equation*}%
which shows that for $t<\xi $ the increments of $Z_{t}$ are mirror images of
the increments of $W_{t}$ with respect to the line of symmetry $M_{t}$ of $%
X_{t}$ and $Y_{t}$, which justifies the name of \emph{mirror coupling}.
\end{remark}

\section{Extension of the mirror coupling\label{Extensions of mirror
coupling}}

The main contribution of the author is the observation that the
mirror coupling introduced above can be extended to the case when
the two reflecting Brownian motion have different state spaces, that
is when $X_{t}$ is a reflecting Brownian motion in $D_{1}$ and
$Y_{t}$ is a reflecting Brownian motion in $D_{2}$. Although the
construction can be carried out in a more general setup (see the
concluding remarks in Section \ref{Extensions and applications}), in
the present section we will restrict to the case when one of the
domains is strictly contained in the other one.

The main result is the following:

\begin{theorem}
\label{main theorem}Let $D_{1,2}\subset \mathbb{R}^{d}$ be smooth bounded
domains (piecewise $C^{2}$-smooth boundary with convex corners in $\mathbb{R}%
^{2}$, or $C^{2}$-smooth boundary in $\mathbb{R}^{d}$, $d\geq 3$ will
suffice) with $\overline{D_{2}}\subset D_{1}$ and $D_{2}$ convex domain, and
let $x\in \bar{D}_{1}$ and $y\in \bar{D}_{2}$ be arbitrarily fixed points.
Given a $d$-dimensional Brownian motion $\left( W_{t}\right) _{t\geq 0}$
starting at $0$ on a probability space $\left( \Omega ,\mathcal{F},P\right) $%
, there exists a strong solution of the following system of stochastic
differential equations%
\begin{eqnarray}
X_{t} &=&x+W_{t}+\int_{0}^{t}\nu _{D_{1}}\left( X_{s}\right) dL_{s}^{X}
\label{1} \\
Y_{t} &=&y+Z_{t}+\int_{0}^{t}\nu _{D_{2}}\left( Y_{s}\right) dL_{s}^{Y}
\label{2} \\
Z_{t} &=&\int_{0}^{t}G\left( Y_{s}-X_{s}\right) dW_{s}  \label{3}
\end{eqnarray}%
or equivalent%
\begin{equation}
Z_{t}=\int_{0}^{t}G\left( \tilde{\Gamma}\left( y+Z\right) _{s}-\Gamma \left(
x+W\right) _{s}\right) dW_{s},  \label{4}
\end{equation}%
where $\Gamma $ and $\tilde{\Gamma}$ denote the corresponding Skorokhod maps
which define the reflecting Brownian motion $X=\Gamma \left( x+W\right) $ in
$D_{1}$, respectively $Y=\tilde{\Gamma}\left( y+Z\right) $ in $D_{2}$, and $%
G:\mathbb{R}^{d}\rightarrow \mathcal{M}_{d\times d}$ denotes the following
modification of the function $G$ defined in the previous section:%
\begin{equation}
G\left( z\right) =\left\{
\begin{tabular}{ll}
$H\left( \frac{z}{\left\vert z\right\vert }\right) ,\qquad $ & if $z\neq 0$
\\
$I,$ & if $z=0$%
\end{tabular}%
\right. .  \label{new def of G}
\end{equation}
\end{theorem}

\begin{remark}
As it will follow from the proof of the theorem, with the choice of $G$
above, in the case $D_{1}=D_{2}=D$ the solution of the equation (\ref{4})
given by the theorem above is the same as the solution of the equation (\ref%
{Diff eq for Z}) considered by the authors in \cite{AtarBurdzy1} (as pointed
out by the authors, the choice of $G\left( 0\right) $ is irrelevant in this
case), and therefore the above theorem is a natural generalization of their
result to the case when the two processes live in different spaces. We will
refer to a solution $X_{t},Y_{t}$ given by the above theorem as a \emph{%
mirror coupling} of reflecting Brownian motions in $D_{1}$, respectively $%
D_{2}$, starting from $\left( x,y\right) \in \overline{D_{1}}\times
\overline{D_{2}}$ with driving Brownian motion $W_{t}$.

As we will see in Section \ref{Extensions and applications}, without
additional assumptions, the solution of (\ref{4}) is not pathwise unique.
This is due to the fact that the stochastic differential equation has a
singularity at the origin (i.e. at times when the coupling occurs); the
general mirror coupling can be thought as depending on a parameter which is
a measure of the stickiness of the coupling: once the processes $X_{t}$ and $%
Y_{t}$ have coupled, they can either move together until one of them hits
the boundary (\emph{sticky }mirror coupling - this is in fact the solution
constructed in the above theorem), or they can immediately split after
coupling (non-sticky mirror coupling), and there is a whole range of
intermediate possibilities (see the discussion at the end of Section \ref%
{Extensions and applications}).

As an application, in Section \ref{Extensions and applications} we
will use the former mirror coupling (the sticky mirror coupling) to
give a unifying proof of Chavel's conjecture on the domain
monotonicity of the Neumann heat kernel for domains $D_{1,2}$
satisfying the ball condition, although the other possible choices
for the mirror coupling might prove useful in other contexts.
\end{remark}

Before carrying out the proof, we begin with some preliminary
remarks which
will allow us to reduce the proof of the above theorem to the case $D_{1}=%
\mathbb{R}^{d}$.

\begin{remark}
The main difference from the case when $D_{1}=D_{2}=D$ considered by the
authors in \cite{AtarBurdzy1} is that after the coupling time $\xi $ the
processes $X_{t}$ and $Y_{t}$ may decouple. For example, if $t\geq \xi $ is
a time when $X_{t}=Y_{t}\in \partial D_{2}$, the process $Y_{t}$ being
conditioned to stay in $\overline{D_{2}}$, receives a push in the direction
of the inward unit normal to the boundary of $D_{2}$, while the process $%
X_{t}$ behaves like a free Brownian motion near this point (we assumed that $%
D_{2}$ is strictly contained in $D_{1}$), and therefore the processes $X$
and $Y$ will drift apart, that is they will \emph{decouple}. Also, as shown
in Section \ref{Extensions and applications}, because the function $G$ has a
discontinuity at the origin, it is possible that the solutions decouple even
inside the domain $D_{2}$, so, without additional assumptions, the mirror
coupling is not uniquely determined (there is no pathwise uniqueness of (\ref%
{4})).
\end{remark}

\begin{remark}
\label{reduction to the case D_1=R}To fix ideas, for an arbitrarily fixed $%
\varepsilon >0$ chosen small enough such that $\varepsilon <\limfunc{dist}%
\left( \partial D_{1},\partial D_{2}\right) $, we consider the sequence $%
\left( \xi _{n}\right) _{n\geq 1}$ of \emph{coupling times} and the sequence
$\left( \tau _{n}\right) _{n\geq 0}$ of times when the processes are $%
\varepsilon $-decoupled ($\varepsilon $\emph{-decoupling times}, or simply
\emph{decoupling times }by an abuse of language) defined inductively by%
\begin{eqnarray*}
\xi _{n} &=&\inf \left\{ t>\tau _{n-1}:X_{t}=Y_{t}\right\} , \\
\tau _{n} &=&\inf \left\{ t>\xi _{n}:\left\vert X_{t}-Y_{t}\right\vert
>\varepsilon \right\} ,
\end{eqnarray*}%
where $\tau _{0}=0$ and $\xi _{1}=\xi $ is the first coupling time.

To construct the general mirror coupling (that is, to prove the existence of
a solution to (\ref{1})-(\ref{3}) above, or equivalent to (\ref{4})), we
proceed as follows.

First note that on the time interval $\left[ 0,\xi \right] $, the arguments
used in the proof of Theorem 2 in \cite{AtarBurdzy1} (pathwise uniqueness
and the existence of a strong solution $Z$ of (\ref{4})) do not rely on the
fact that $D_{1}=D_{2}$, hence the same arguments can be used to prove the
existence of a strong solution of (\ref{4}) on the time interval $[0,\xi
_{1}]=\left[ 0,\xi \right] $. Indeed, given $W_{t}$, (\ref{1}) has a strong
solution which is pathwise unique (the reflecting Brownian motion $X_{t}$ in
$D_{1}$), and therefore the proof of pathwise uniqueness and the existence
of a strong solution of (\ref{4}) is the same as in \cite{AtarBurdzy1}
considering $D=D_{2}$. Also note that as pointed by the authors, the value $%
G\left( 0\right) $ is irrelevant in their proof, since the problem is
constructing the processes until they meet, that is for $Y_{t}-X_{t}\neq 0$,
for which the definition of $G$ coincides with (\ref{new def of G}).

Next, assuming the existence of a strong solution to (\ref{4}) on $\left[
\xi _{1},\tau _{1}\right] $ (and therefore on $\left[ 0,\tau _{1}\right] $),
since at time $\tau _{1}$ the processes are $\varepsilon >0$ units apart, we
can apply again the results in \cite{AtarBurdzy1} (with $\tilde{B}%
_{t}=B_{t+\tau _{1}}-B_{\tau _{1}}$ and starting points $X_{\tau _{1}}$ and $%
Y_{\tau _{1}}$) in order to obtain a strong solution of (\ref{4}) on the
time interval $\left[ \tau _{1},\xi _{2}\right] $, and therefore by patching
we obtain the existence of a strong solution of (\ref{4}) on the time
interval $\left[ 0,\xi _{2}\right] $.

For an arbitrarily fixed $t>0$, since only a finite number of
coupling/decoupling times $\xi _{n}$ and $\tau _{n}$ can occur in the time
interval $\left[ 0,t\right] $ (we use here the fact that $D_{2}$ is strictly
contained in $D_{1}$), it follows that there exists a strong solution to (%
\ref{4}) on $\left[ 0,t\right] $ for any $t>0$ (and therefore on $[0,\infty
) $), provided we show the existence of a strong solution of (\ref{4}) on $%
\left[ \xi _{n},\tau _{n}\right] $, $n\geq 1$.

In order to prove this claim, it suffices therefore to show that for any
starting points $x=y\in \bar{D}_{2}$ of the mirror coupling, there exists a
strong solution to (\ref{4}) until the $\varepsilon $-decoupling time $\tau
_{1}$. Since $\varepsilon <\limfunc{dist}\left( \partial D_{1},\partial
D_{2}\right) $, it follows that the process $X_{t}$ cannot reach the
boundary $\partial D_{1}$ before the $\varepsilon $-decoupling time $\tau
_{1}$, and therefore we can consider that $X_{t}$ is a free Brownian motion
in $\mathbb{R}^{d}$, that is we can reduce the proof of Theorem \ref{main
theorem} to the case when $D_{1}=\mathbb{R}^{d}$.
\end{remark}

We will first give the proof of the in the $1$-dimensional case, then we
will extend the construction to polygonal domains in $\mathbb{R}^{d}$, and
we will conclude with the proof in the general case.

\subsection{The $1$-dimensional case\label{1-dimensional case}}

From Remark \ref{reduction to the case D_1=R} it follows that in order to
construct the mirror coupling in the $1$-dimensional case, it suffices to
consider $D_{1}=\mathbb{R}$ and $D_{2}=\left( 0,a\right) $ and to construct
a strong solution for $t\leq \tau _{1}=\inf \left\{ s>0:\left\vert
X_{s}-Y_{s}\right\vert >\varepsilon \right\} $ of the following system:%
\begin{eqnarray}
X_{t} &=&x+W_{t}  \label{1-1dim} \\
Y_{t} &=&x+Z_{t}+L_{t}^{Y}  \label{2-1-dim} \\
Z_{t} &=&\int_{0}^{t}G\left( Y_{s}-X_{s}\right) dW_{s}  \label{3-1-dim}
\end{eqnarray}%
for an arbitrary choice $x\in \left[ 0,a\right] $ of the starting point of
the mirror coupling, where $\varepsilon \in \left( 0,a\right) $ is
arbitrarily small, $\left( W_{t}\right) _{t\geq 0}$ is a $1$-dimensional
Brownian motion starting at $W_{0}=0$ and the function $G:\mathbb{%
R\rightarrow \mathcal{M}}_{1\times 1}\equiv \mathbb{R}$ is given in this
case by%
\begin{equation*}
G\left( x\right) =\left\{
\begin{tabular}{ll}
$-1,\qquad $ & if $x\neq 0$ \\
$+1,$ & if $x=0$%
\end{tabular}%
\right. .
\end{equation*}

\begin{remark}
\label{remark on Tanaka formula}Before proceeding with the proof, it is
worth mentioning that the heart of the construction is Tanaka's formula. To
see this, consider for the moment $a=\infty $, and note that Tanaka formula
\begin{equation*}
\left\vert x+W_{t}\right\vert =x+\int_{0}^{t}\limfunc{sgn}\left(
x+W_{s}\right) dW_{s}+L_{t}^{0}\left( x+W\right)
\end{equation*}%
gives a representation of the reflecting Brownian motion $\left\vert
x+W_{t}\right\vert $ in which the increments of the martingale part of $%
\left\vert x+W_{t}\right\vert $ are the increments of $W_{t}$ when $%
x+W_{t}\in \lbrack 0,\infty )$, respectively the opposite (minus) of the
increments of $W_{t}$ in the opposite case ($L_{t}^{0}\left( x+W\right) $
denotes here the local time at $0$ of $x+W_{t}$).

Noting that the condition $x+W_{t}\in \lbrack 0,\infty )$ is the same as $%
\left\vert x+W_{t}\right\vert =x+W_{t}$, from the definition of the function
$G$ it follows that the above can be written in the form%
\begin{equation*}
\left\vert x+W_{t}\right\vert =x+\int_{0}^{t}G\left( \left\vert
x+W_{s}\right\vert -\left( x+W_{s}\right) \right) dW_{s}+L_{t}^{x+W},
\end{equation*}%
which shows that a strong solution to (\ref{1-1dim}) - (\ref{3-1-dim}) above
(in the case $a=\infty $) is given explicitly by $X_{t}=x+W_{t}$ and $%
Y_{t}=\left\vert x+W_{t}\right\vert $ (and $Z_{t}=\int_{0}^{t}\limfunc{sgn}%
\left( x+W_{s}\right) dW_{s}$).
\end{remark}

We have the following:

\begin{proposition}
\label{Prop 1-dim case}Given a $1$-dimensional Brownian motion $\left(
W_{t}\right) _{t\geq 0}$ starting at $W_{0}=0$, a strong solution to (\ref%
{1-1dim}) -- (\ref{3-1-dim}) above for $t<\tau _{1}=\inf \left\{
s>0:\left\vert X_{s}-Y_{s}\right\vert >\varepsilon \right\} $ is given by%
\begin{equation*}
\left\{
\begin{array}{l}
X_{t}=x+W_{t} \\
Y_{t}=\left\vert a-\left\vert x+W_{t}-a\right\vert \right\vert \\
Z_{t}=\int_{0}^{t}\limfunc{sgn}\left( W_{s}\right) \limfunc{sgn}\left(
a-W_{s}\right) dW_{s}%
\end{array}%
\right. ,
\end{equation*}%
where
\begin{equation*}
\limfunc{sgn}\left( x\right) =\left\{
\begin{tabular}{ll}
$+1,\qquad $ & if $x\geq 0$ \\
$-1,$ & if $x<0$%
\end{tabular}%
\right. .
\end{equation*}
\end{proposition}

\begin{proof}
Since $\varepsilon <a$, it follows that for $t\leq \tau _{1}$ we have $%
X_{t}=x+W_{t}\in \left( -a,2a\right) $, and therefore
\begin{equation}
Y_{t}=\left\vert a-\left\vert x+W_{t}-a\right\vert \right\vert =\left\{
\begin{tabular}{ll}
$-\left( x+W_{t}\right) ,$ & $x+W_{t}\in \left( -a,0\right) $ \\
$x+W_{t},$ & $x+W_{t}\in \lbrack 0,a]$ \\
$2a-x-W_{t},$ & $x+W_{t}\in (a,2a)$%
\end{tabular}%
\right. .  \label{aux2}
\end{equation}

Applying the Tanaka-It\^{o} formula to the function $f\left( z\right)
=\left\vert a-\left\vert z-a\right\vert \right\vert $ and to the Brownian
motion $X_{t}=x+W_{t}$, for $t\leq \tau _{1}$ we obtain%
\begin{eqnarray*}
Y_{t} &=&x+\int_{0}^{t}\limfunc{sgn}\left( x+W_{s}\right) \limfunc{sgn}%
\left( a-x-W_{s}\right) d\left( x+W_{s}\right) +L_{t}^{0}-L_{t}^{a} \\
&=&x+\int_{0}^{t}\limfunc{sgn}\left( x+W_{s}\right) \limfunc{sgn}\left(
a-x-W_{s}\right) dW_{s}+\int_{0}^{t}\nu _{D_{2}}\left( Y_{s}\right) d\left(
L_{s}^{0}+L_{s}^{a}\right) ,
\end{eqnarray*}%
where $L_{t}^{0}=\sup_{s\leq t}\left( x+W_{s}\right) ^{-}$ and $%
L_{t}^{a}=\sup_{s\leq t}\left( x+W_{s}-a\right) ^{+}$ are the local times of
$x+W_{t}$ at $0$, respectively at $a$, and $\nu _{D_{2}}\left( 0\right) =+1$%
, $\nu _{D_{2}}\left( a\right) =-1$.

From (\ref{aux2}) and the definition of $G$ we have%
\begin{eqnarray*}
\limfunc{sgn}\left( x+W_{s}\right) \limfunc{sgn}\left( a-x-W_{s}\right)
&=&\left\{
\begin{tabular}{ll}
$-1,$ & $x+W_{s}\in \left( -a,0\right) $ \\
$+1,$ & $x+W_{s}\in \lbrack 0,a]$ \\
$-1,$ & $x+W_{s}\in (a,2a)$%
\end{tabular}%
\right. \\
&=&\left\{
\begin{tabular}{ll}
$+1,$ & $X_{s}=Y_{s}$ \\
$-1,$ & $X_{s}\neq Y_{s}$%
\end{tabular}%
\right. \\
&=&G\left( Y_{s}-X_{s}\right) ,
\end{eqnarray*}%
and therefore the previous formula can be written equivalently%
\begin{equation*}
Y_{t}=x+Z_{t}+\int_{0}^{t}\nu _{D_{2}}\left( Y_{s}\right) dL_{s}^{Y},
\end{equation*}%
where
\begin{equation*}
Z_{t}=\int_{0}^{t}G\left( Y_{s}-X_{s}\right) dW_{s}
\end{equation*}%
and $L_{t}^{Y}=L_{t}^{0}+L_{t}^{a}$ is a continuous nondecreasing process
which increases only when $x+W_{t}\in \left\{ 0,a\right\} $, that is only
when $Y_{t}\in \partial D_{2}$, which concludes the proof.
\end{proof}

\subsection{The case of polygonal domains\label{Case of polygonal domains}}

In this section we will consider the case when $D_{2}\subset D_{1}\subset
\mathbb{R}^{d}$ are convex polygonal domains (convex domains bounded by
hyperplanes in $\mathbb{R}^{d}$). From Remark \ref{reduction to the case
D_1=R} it follows that we can consider $D_{1}=\mathbb{R}^{d}$ and therefore
it suffices to prove the existence of a strong solution to

\begin{eqnarray}
X_{t} &=&X_{0}+W_{t}  \label{1-2dim} \\
Y_{t} &=&Y_{0}+Z_{t}+\int_{0}^{t}\nu _{D_{2}}\left( Y_{s}\right) dL_{s}^{Y}
\label{2-2dim} \\
Z_{t} &=&\int_{0}^{t}G\left( Y_{s}-X_{s}\right) dW_{s}  \label{3-2dim}
\end{eqnarray}%
or equivalently%
\begin{equation}
Z_{t}=\int_{0}^{t}G\left( \tilde{\Gamma}\left( Y_{0}+Z\right)
_{s}-X_{0}-W_{s}\right) dW_{s},  \label{4-2dim}
\end{equation}%
where $W_{t}$ is a $d$-dimensional Brownian motion starting at $W_{0}=0$ and
$X_{0}=Y_{0}\in \bar{D}_{2}$.

The construction relies on the skew product representation of Brownian
motion in spherical coordinates, that is%
\begin{equation}
X_{t}=R_{t}\Theta _{\sigma _{t}},  \label{skew product representation of BM}
\end{equation}%
where $R_{t}=\left\vert X_{t}\right\vert \in \limfunc{BES}\left( d\right) $
is a Bessel process of order $d$ and $\Theta _{t}\in \limfunc{BM}\left(
S^{d-1}\right) $ is an independent Brownian motion on the unit sphere $%
S^{d-1}$ in $\mathbb{R}^{d}$, run at speed%
\begin{equation}
\sigma _{t}=\int_{0}^{t}\frac{1}{R_{s}^{2}}ds,
\label{time change for skew product repr of BM}
\end{equation}%
which depends only on $R_{t}$.

\begin{remark}
One way to construct the Brownian motion $\Theta _{t}=\Theta _{t}^{d-1}$ on
the unit sphere $S^{d-1}\subset \mathbb{R}^{d}$ is to proceed inductively on
$d\geq 2$, using the skew product representation of Brownian motion on the
sphere $\Theta _{t}^{d-1}\in S^{d-1}$ (see \cite{Ito-McKean})%
\begin{equation*}
\Theta _{t}^{d-1}=\left( \cos \theta _{t}^{1},\sin \theta _{t}^{1}\Theta
_{\alpha _{t}}^{d-2}\right)
\end{equation*}
where $\theta ^{1}\in \limfunc{LEG}\left( d-1\right) $ is a Legendre process
of order $d-1$ on $\left[ 0,\pi \right] $, and $\Theta _{t}^{d-2}\in S^{d-2}$
is an independent Brownian motion on $S^{d-2}$, run at speed
\begin{equation*}
\alpha _{t}=\int_{0}^{t}\frac{1}{\sin ^{2}\theta _{s}^{1}}ds.
\end{equation*}

Therefore, considering independent processes $\theta _{t}^{1},\ldots \theta
_{t}^{d-1}$, where $\theta ^{i}\in \limfunc{LEG}\left( d-i\right) $ on $%
\left[ 0,\pi \right] $ ($i=\overline{1,d-2}$) and $\theta _{t}^{d-1}$ a $1$%
-dimensional Brownian ($\Theta _{t}^{1}=\left( \cos \theta _{t}^{1},\sin
\theta _{t}^{1}\right) \in S^{1}$ is a Brownian motion on $S^{1}$), we have%
\begin{equation*}
\Theta _{t}^{d-1}=\left( \cos \theta _{t}^{1},\sin \theta _{t}^{1}\cos
\theta _{t}^{2},\sin \theta _{t}^{1}\sin \theta _{t}^{2}\cos \theta
_{t}^{3},\ldots ,\sin \theta _{t}^{1}\cdots \sin \theta _{t}^{d-1}\sin
\theta _{t}^{d-1}\right) ,
\end{equation*}%
or equivalent, in spherical coordinates, $\Theta _{t}^{d-1}\in S^{d-1}$ is
given by
\begin{equation}
\Theta _{t}^{d-1}=\left( \theta _{t}^{1},\ldots ,\theta _{t}^{d-2},\theta
_{t}^{d-1}\right) .
\label{skew product BM on the sphere in spherical coordinates}
\end{equation}
\end{remark}

To construct the solution we first consider first the case when $D_{2}$ is a
half-space $\mathcal{H}_{d}^{+}=\left\{ \left( z^{1},\ldots ,z^{d}\right)
\in \mathbb{R}^{d}:z^{d}>0\right\} $.

Given an angle $\varphi \in \mathbb{R}$, we introduce the rotation matrix $%
R\left( \varphi \right) \in \mathcal{M}_{d\times d}$ which leaves invariant
the first $d-2$ coordinates and rotates clockwise by the angle $\alpha $ the
remaining $2$ coordinates, that is%
\begin{equation}
R\left( \alpha \right) =\left(
\begin{array}{llllcc}
1 &  & 0 &  & 0 & 0 \\
& \ddots &  &  & \cdots & \cdots \\
0 &  & 1 &  & 0 & 0 \\
0 & \cdots & 0 &  & \cos \varphi & -\sin \varphi \\
0 & \cdots & 0 &  & \sin \varphi & \cos \varphi%
\end{array}%
\right)  \label{rotation matrix}
\end{equation}

We have the following:

\begin{lemma}
\label{lemma for a half-space}Let $D_{2}=\mathcal{H}_{d}^{+}=\left\{ \left(
z^{1},\ldots ,z^{d}\right) \in \mathbb{R}^{d}:z^{d}>0\right\} $ and assume
that
\begin{equation}
Y_{0}=R\left( \varphi _{0}\right) X_{0}  \label{condition for a half-plane}
\end{equation}%
for some $\varphi _{0}\in \mathbb{R}$.

Consider the reflecting Brownian motion $\tilde{\theta}_{t}^{d-1}$ on $\left[
0,\pi \right] $ with driving Brownian motion $\theta _{t}^{d-1}$, where $%
\theta _{t}^{d-1}$ is the $\left( d-1\right) $ spherical coordinate of $%
G\left( Y_{0}-X_{0}\right) X_{t}$, given by (\ref{skew product
representation of BM}) -- (\ref{skew product BM on the sphere in spherical
coordinates}) above, that is:
\begin{equation*}
\tilde{\theta}_{t}^{d-1}=\theta _{t}^{d-1}+L_{t}^{0}\left( \tilde{\theta}%
^{d-1}\right) -L_{t}^{\pi }\left( \tilde{\theta}^{d-1}\right) ,\qquad t\geq
0,
\end{equation*}%
and $L_{t}^{0}\left( \tilde{\theta}^{d-1}\right) $, $L_{t}^{\pi }\left(
\tilde{\theta}^{d-1}\right) $ represent the local times of $\tilde{\theta}%
^{d-1}$ at $0$, respectively at $\pi $.

A strong solution to (\ref{1-2dim}) - (\ref{3-2dim}) above is explicitly
given by%
\begin{equation}
Y_{t}=\left\{
\begin{tabular}{ll}
$R\left( \varphi _{t}\right) G\left( Y_{0}-X_{0}\right) X_{t},$ & $t<\xi $
\\
$\left\vert X_{t}\right\vert _{d},$ & $t\geq \xi $%
\end{tabular}%
\right.  \label{eq of the solution in a half-plane}
\end{equation}%
where $\xi =\inf \left\{ t>0:X_{t}=Y_{t}\right\} $ is the coupling time, the
rotation angle $\varphi _{t}$ is given by
\begin{equation*}
\varphi _{t}=L_{t}^{0}\left( \tilde{\theta}^{d-1}\right) -L_{t}^{\pi }\left(
\tilde{\theta}^{d-1}\right) ,\qquad t\geq 0,
\end{equation*}%
and for $z=\left( z^{1},z^{2}\ldots ,z^{d}\right) \in \mathbb{R}^{d}$ we
denoted by $\left\vert z\right\vert _{d}=\left( z^{1},z^{2},\ldots
,\left\vert z^{d}\right\vert \right) $.
\end{lemma}


\begin{figure}[thb]
\begin{center}
\includegraphics[scale=0.7]{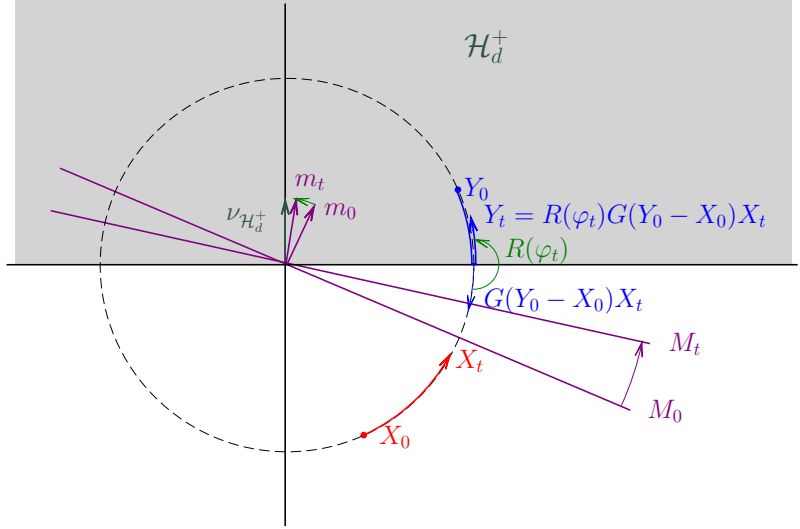}
\caption{The mirror coupling of a free Brownian motion $X_{t}$ and a
reflecting Brownian motion $Y_{t}$ in the half-space
$\mathcal{H}_{d}^{+}$.} \label{fig1}
\end{center}
\end{figure}

\begin{proof}
Recall that for $m\in \mathbb{R}^{d}-\left\{ 0\right\} $, $G\left( m\right)
v $ denotes the mirror image (symmetric) of $v\in \mathbb{R}^{d}$ in the
hyperplane through the origin, perpendicular to $m$.

By It\^{o} formula, we have%
\begin{equation}
Y_{t\wedge \xi }=Y_{0}+\int_{0}^{t\wedge \xi }R\left( \varphi _{s}\right)
G\left( Y_{0}-X_{0}\right) dX_{s}+\int_{0}^{t\wedge \xi }R\left( \varphi
_{s}+\frac{\pi }{2}\right) G\left( Y_{0}-X_{0}\right) dL_{s}  \label{aux1}
\end{equation}

Note that the composition $R\circ G$ (a symmetry followed by a rotation) is
a symmetry, and since $\left\vert Y_{t}\right\vert =\left\vert
X_{t}\right\vert $ for all $t\geq 0$, it follows that $X_{t}$ and $Y_{t}$
are symmetric with respect to a hyperplane passing through the origin for
all $t\leq \xi $; therefore, from the definition (\ref{new def of G}) of $G$
it follows that we have $Y_{t}=G\left( Y_{t}-X_{t}\right) X_{t}$ for all $%
t\leq \xi $.

Also note that when $L_{s}^{0}\left( \tilde{\theta}^{d-1}\right) $
increases, $Y_{s}\in \partial D_{2}$ and we have
\begin{equation*}
R\left( \varphi _{s}+\frac{\pi }{2}\right) G\left( Y_{0}-X_{0}\right)
X_{s}=R\left( \frac{\pi }{2}\right) Y_{s}=\nu _{D_{2}}\left( Y_{s}\right)
\end{equation*}%
and if $L_{s}^{\pi }\left( \tilde{\theta}^{d-1}\right) $ increases, $%
Y_{s}\in \partial D_{2}$ and we have
\begin{equation*}
R\left( \varphi _{s}+\frac{\pi }{2}\right) G\left( Y_{0}-X_{0}\right)
X_{s}=R\left( \frac{\pi }{2}\right) Y_{s}=-\nu _{D_{2}}\left( Y_{s}\right) .
\end{equation*}%
It follows that the relation (\ref{aux1}) above can be written in the form%
\begin{equation*}
Y_{t\wedge \xi }=Y_{0}+\int_{0}^{t\wedge \xi }G\left( Y_{s}-X_{s}\right)
dX_{s}+\int_{0}^{t\wedge \xi }\nu _{D_{2}}\left( Y_{s}\right) dL_{s}^{Y},
\end{equation*}%
where $L_{t}^{Y}=L_{t}^{0}\left( \tilde{\theta}^{d-1}\right) +L_{t}^{\pi
}\left( \tilde{\theta}^{d-1}\right) $ is a continuous non-decreasing process
which increases only when $Y_{t}\in \partial D_{2}$, and therefore $Y_{t}$
given by (\ref{eq of the solution in a half-plane}) is a strong solution to (%
\ref{1-2dim}) - (\ref{3-2dim}) for $t\leq \xi $.

For $t\geq \xi $, we have $Y_{t}=\left\vert X_{t}\right\vert _{d}=\left(
X_{t}^{1},X_{t}^{2},\ldots ,\left\vert X_{t}^{d}\right\vert \right) $, and
similar to the $1$-dimensional case, by Tanaka formula we obtain:%
\begin{eqnarray}
Y_{t\vee \xi } &=&Y_{\xi }+\int_{\xi }^{t\vee \xi }\left( 1,\ldots ,1,%
\limfunc{sgn}\left( X_{s}^{d}\right) \right) dX_{s}+\int_{\xi }^{t\vee \xi
}\left( 0,0,\ldots ,1\right) L_{t}^{0}\left( X^{d}\right)  \label{aux3} \\
&=&Y_{\xi }+\int_{\xi }^{t\vee \xi }G\left( Y_{s}-X_{s}\right)
dX_{s}+\int_{\xi }^{t\vee \xi }\nu _{D_{2}}\left( Y_{s}\right) L_{t}^{Y},
\notag
\end{eqnarray}%
since
\begin{eqnarray*}
G\left( Y_{s}-X_{s}\right) &=&\left\{
\begin{tabular}{ll}
$\left( 1,\ldots ,1,+1\right) ,$ & $X_{s}=Y_{s}$ \\
$\left( 1,\ldots ,1,-1\right) ,$ & $X_{s}\neq Y_{s}$%
\end{tabular}%
\right. \\
&=&\left\{
\begin{tabular}{ll}
$\left( 1,\ldots ,1,+1\right) ,$ & $X_{s}^{d}\geq 0$ \\
$\left( 1,\ldots ,1,-1\right) ,$ & $X_{s}^{d}<0$%
\end{tabular}%
\right. \\
&=&\left( 1,\ldots ,1,\limfunc{sgn}\left( X_{s}^{d}\right) \right) .
\end{eqnarray*}

$L_{t}^{Y}=L_{t}^{0}\left( X^{d}\right) $ in (\ref{aux3}) is a continuous
non-decreasing process which increases only when $Y_{t}\in \partial D_{2}$
(we denoted by $L_{t}^{0}\left( X^{d}\right) $ the local time at $0$ of the
last cartesian coordinate $X^{d}$ of $X$), which shows that $Y_{t}$ also
solves (\ref{1-2dim}) - (\ref{3-2dim}) for $t\geq \xi $, and therefore $%
Y_{t} $ is a strong solution to (\ref{1-2dim}) - (\ref{3-2dim}) for $t\geq 0$%
, concluding the proof.
\end{proof}

Consider now the case of a general polygonal domain $D_{2}\subset \mathbb{R}%
^{d}$. We will show that a strong solution to (\ref{1-2dim}) - (\ref{3-2dim}%
) can be constructed from the previous lemma, by choosing the
appropriate coordinate system.

Consider the times $\left( \sigma _{n}\right) _{n\geq 0}$ at which the
solution $Y_{t}$ hits different bounding hyperplanes of $\partial D_{2}$,
that is $\sigma _{0}=\inf \left\{ s\geq 0:Y_{s}\in \partial D_{2}\right\} $
and inductively%
\begin{equation}
\sigma _{n+1}=\inf \left\{ t\geq \sigma _{n}:%
\begin{array}{c}
Y_{t}\in \partial \mathcal{D}_{2}\text{ and }Y_{t}\text{,}Y_{\sigma _{n}}%
\text{ belong to different\footnote{Since $2$-dimensional Brownian
motion does not hit points a.s., the $d$-dimensional Brownian motion
$Y_{t}$ does not hit the edges of $D_{2}$ ($ \left( d-2\right)
$-dimensional hyperplanes in $\mathbb{R}^{d}$) a.s., thus
there is no ambiguity in the definition.} } \\
\text{bounding hyperplanes of }\partial D_{2}%
\end{array}%
\right\} .
\end{equation}

If $X_{0}=Y_{0}\in \partial D_{2}$ belong to a certain bounding hyperplane
of $D_{2}$, we can chose the coordinate system so that this hyperplane is $%
\mathcal{H}_{d}=\left\{ \left( z^{1},\ldots ,z^{d}\right) \in \mathbb{R}%
^{d}:z^{d}=0\right\} $ and $D_{2}\subset \mathcal{H}_{d}^{+}$, and we let $%
\mathcal{H}_{d}$ be any bounding hyperplane of $D_{2}$ otherwise.

Then, on the time interval $[\sigma _{0},\sigma _{1})$, the strong solution
to (\ref{1-2dim}) - (\ref{3-2dim}) is given explicitly by (\ref{eq of the
solution in a half-plane}) in Lemma \ref{lemma for a half-space}.

If $\sigma _{1}<\infty $, we distinguish two cases: $X_{\sigma
_{1}}=Y_{\sigma _{1}}$ and $X_{\sigma _{1}}\neq Y_{\sigma _{1}}$. Let $%
\mathcal{H}$ denote the bounding hyperplane of $D$ which contains $Y_{\sigma
_{1}}$, and let $\nu _{\mathcal{H}}$ denote the unit normal to $\mathcal{H}$
pointing inside $D_{2}$.

If $X_{\sigma _{1}}=Y_{\sigma _{1}}\in \mathcal{H}$, choosing again the
coordinate system conveniently, we may assume that $\mathcal{H}$ is the
hyperplane is $\mathcal{H}_{d}=\left\{ \left( z^{1},\ldots ,z^{d}\right) \in
\mathbb{R}^{d}:z^{d}=0\right\} $, and on the time interval $[\sigma
_{1},\sigma _{2})$ the coupling $\left( X_{\sigma _{1}+t},Y_{\sigma
_{1}+t}\right) _{t\in \lbrack 0,\sigma _{2}-\sigma _{1})}$ is given again by
Lemma \ref{lemma for a half-space}.

If $X_{\sigma _{1}}\neq Y_{\sigma _{1}}\in \mathcal{H}$, in order to apply
the lemma, we have to show that we can choose the coordinate system so that
the condition (\ref{condition for a half-plane}) holds. If $Y_{\sigma
_{1}}-X_{\sigma _{1}}$ is a vector perpendicular to $\mathcal{H}$, by
choosing the coordinate system so that $\mathcal{H}=\mathcal{H}_{d}=\left\{
\left( z^{1},\ldots ,z^{d}\right) \in \mathbb{R}^{d}:z^{d}=0\right\} $, the
problem reduces to the $1$-dimensional case (the first $d-1$ coordinates of $%
X$ and $Y$ are the same), and it can be handled as in Proposition \ref{Prop
1-dim case} by the Tanaka formula. The proof being similar, we omit it.

If $X_{\sigma _{1}}\neq Y_{\sigma _{1}}\in \mathcal{H}$ and $Y_{\sigma
_{1}}-X_{\sigma _{1}}$ is not orthogonal to $\mathcal{H}$, consider $\tilde{X%
}_{\sigma _{1}}=\limfunc{pr}_{\mathcal{H}}X_{\sigma _{1}}$ the projection of
$X_{\sigma _{1}}$ onto $\mathcal{H}$, and therefore $\tilde{X}_{\sigma
_{1}}\neq Y_{\sigma _{1}}$. The plane of symmetry of $X_{\sigma _{1}}$ and $%
Y_{\sigma _{1}}$ intersects the line determined by $\tilde{X}_{\sigma _{1}}$
and $Y_{\sigma _{1}}$ at a point, and we consider this point as the origin
of the coordinate system (note that the intersection cannot be empty, for
otherwise the vectors $Y_{\sigma _{1}}-X_{\sigma _{1}}$ and $Y_{\sigma _{1}}-%
\tilde{X}_{\sigma _{1}}$ were parallel, which is impossible since then $%
Y_{\sigma _{1}}-X_{\sigma _{1}},Y_{\sigma _{1}}-\tilde{X}_{\sigma
_{1}}$ and $Y_{\sigma _{1}}-\tilde{X}_{\sigma _{1}},X_{\sigma
_{1}}-\tilde{X}_{\sigma _{1}}$ were perpendicular pairs of vectors,
contradicting $\tilde{X}_{\sigma _{1}}\neq Y_{\sigma _{1}}$ - see
Figure \ref{fig2}).


\begin{figure}[thb]
\begin{center}
\includegraphics[scale=0.9]{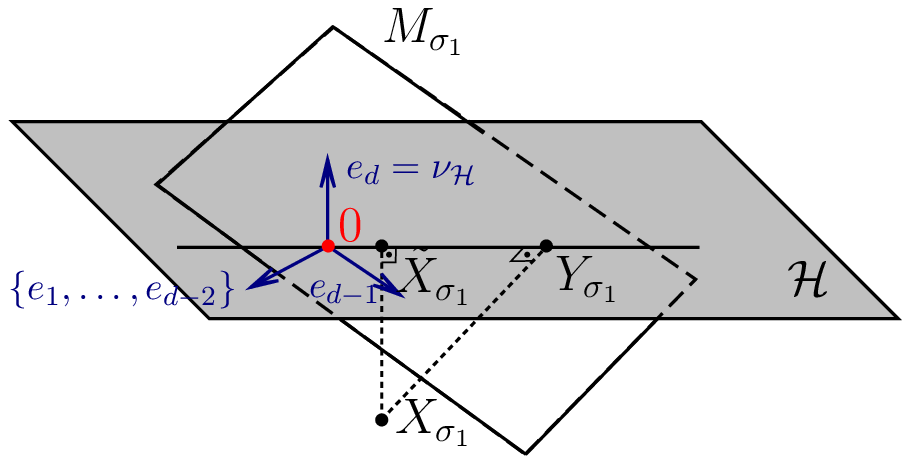}
\caption{Construction of an appropriate coordinate system.}
\label{fig2}
\end{center}
\end{figure}

Choose a orthonormal basis $\left\{ e_{1},\ldots ,e_{d}\right\} $ in $%
\mathbb{R}^{d}$ such that $e_{d}=\nu _{\mathcal{H}}$ is the normal
vector to $\mathcal{H}$ pointing inside $D_{2}$,
$e_{d-1}=\frac{1}{\left\vert Y_{\sigma _{1}}-X_{\sigma
_{1}}\right\vert }\left( Y_{\sigma _{1}}-X_{\sigma _{1}}\right) $ is
a unit vector lying in the $2$-dimensional plane determined by the
origin and the vectors $e_{d}$ and $Y_{\sigma _{1}}-X_{\sigma
_{1}}$, and $\left\{ e_{1},\ldots ,e_{d-2}\right\} $ is a
completion of $\left\{ e_{d-1},e_{d}\right\} $ to a orthonormal basis in $%
\mathbb{R}^{d}$ (see Figure \ref{fig2}).

Note that by the construction, the vectors $e_{1},\ldots ,e_{d-2}$ are
orthogonal to the $2$-dimensional hyperplane containing the origin and the
points $X_{\sigma _{1}}$ and $Y_{\sigma _{1}}$, and therefore $X_{\sigma
_{1}}$ and $Y_{\sigma _{1}}$ have the same (zero) first $d-2$ coordinates;
also, since $X_{\sigma _{1}}$ and $Y_{\sigma _{1}}$ are at the same distance
from the origin, it follows that $Y_{\sigma _{1}}$ can be obtained from $%
X_{\sigma _{1}}$ by a rotation which leaves invariant the first $d-2$
coordinates, which shows that the condition (\ref{condition for a half-plane}%
) of Lemma \ref{lemma for a half-space} is satisfied.

Since by construction the bounding hyperplane $\mathcal{H}$ of $D_{2}$ at $%
Y_{\sigma _{1}}$ is given by $\mathcal{H}_{d}=\left\{ \left( z^{1},\ldots
,z^{d}\right) \in \mathbb{R}^{d}:z^{d}=0\right\} $ and $D_{2}\subset
\mathcal{H}_{d}^{+}=\left\{ \left( z^{1},\ldots ,z^{d}\right) \in \mathbb{R}%
^{d}:z^{d}>0\right\} $, we can apply Lemma \ref{lemma for a half-space} and
deduce that on the time interval $[\sigma _{1},\sigma _{2})$ a solution to (%
\ref{1-2dim}) - (\ref{3-2dim}) is given by $\left( X_{\sigma
_{1}+t},Y_{\sigma _{1}+t}\right) _{t\in \lbrack 0,\sigma _{2}-\sigma _{1})}$.

Repeating the above argument we can construct inductively (in appropriate
coordinate systems) the solution to (\ref{1-2dim}) - (\ref{3-2dim}) on any
time interval $[\sigma _{n},\sigma _{n+1})$, $n\geq 1$, therefore obtaining
a strong solution to (\ref{1-2dim}) - (\ref{3-2dim}) defined for $t\geq 0$.

We summarize the above discussion in the following:

\begin{theorem}
\label{coupling in polygonal domains}If $D_{2}\subset \mathbb{R}^{d}$ is a
convex polygonal domain, for any $X_{0}=Y_{0}\in \bar{D}_{2}$, there exists
a strong solution to (\ref{1-2dim}) - (\ref{3-2dim}) above.

Moreover, between successive hits of different bounding hyperplanes
of $D_{2}$ (i.e. on each time interval $[\sigma _{n},\sigma _{n+1})$
in the notation above) and for an appropriately chosen coordinate
system, the solution is given by Lemma \ref{lemma for a half-space}.
\end{theorem}

We will refer to the solution $\left( X_{t},Y_{t}\right) _{t\geq 0}$
constructed in the previous theorem as a \emph{mirror coupling }of
reflecting Brownian motions in $\left( \mathbb{R}^{d},D_{2}\right) $ with
starting point $X_{0}=Y_{0}\in \overline{D_{2}}$.

If $X_{t}\neq Y_{t},$ the hyperplane $M_{t}$ of symmetry between $X_{t}$ and
$Y_{t}$, passing through $\frac{X_{t}+Y_{t}}{2}$ with normal $m_{t}=\frac{1}{%
\left\vert Y_{t}-X_{t}\right\vert }\left( Y_{t}-X_{t}\right) $, will be
referred to as the \emph{mirror of the coupling}. For definiteness, when $%
X_{t}=Y_{t}$ we let $M_{t}$ denote any hyperplane passing through $%
X_{t}=Y_{t}$, for example we choose $M_{t}$ such that it is a left
continuous function with respect to $t$.

Some of the properties of the mirror coupling are contained in the following:

\begin{proposition}
\label{properties of mirror coupling in polygonal domains}If $D_{2}\subset
\mathbb{R}^{d}$ is a convex polygonal domain, for any $X_{0}=Y_{0}\in \bar{D}%
_{2}$, the mirror coupling given by the previous theorem has the following
properties:

\begin{enumerate}
\item[i)] If the reflection takes place in the boundary hyperplane $\mathcal{%
H}$ of $D_{2}$ with inward unitary normal $\nu _{\mathcal{H}}$, then the
angle $\angle (m_{t};\nu _{\mathcal{H}})$ decreases monotonically to zero.

\item[ii)] When processes are not coupled, the mirror $M_{t}$ lies outside $%
D_{2}$.

\item[iii)] Coupling can take place precisely when $X_{t}\in \partial D_{2}$%
. Moreover, if $X_{t}\in D_{2}$, then $X_{t}=Y_{t}$.

\item[iv)] If $D_{\alpha }\subset D_{\beta }$ are two polygonal domains and $%
(Y_{t}^{\alpha };X_{t})$, $(Y_{t}^{\beta };X_{t})$ are the corresponding
mirror coupling starting from $x\in \overline{D_{\alpha }}$, for any $t>0$
we have%
\begin{equation}
\sup_{s\leq t}\left\vert Y_{s}^{\alpha }-Y_{s}^{\beta }\right\vert \leq
\limfunc{Dist}\left( D^{\alpha },D^{\beta }\right) :=\max_{\substack{ %
x_{\alpha }\in \partial D_{\alpha },x_{\beta }\in \partial D_{\beta }  \\ %
\left( x_{\beta }-x_{\alpha }\right) \cdot \nu _{D_{\alpha }}\left(
x_{\alpha }\right) \leq 0}}\left\vert x_{\alpha }-x_{\beta }\right\vert .
\end{equation}
\end{enumerate}
\end{proposition}

\begin{proof}
i) In the notation of Theorem \ref{coupling in polygonal domains}, on the
time interval $[\sigma _{0},\sigma _{1})$ we have $Y_{t}=X_{t}$ so $\angle
\left( m_{t},\nu _{\mathcal{H}}\right) =0$, and the claim is verified in
this case.

On an arbitrary time interval $[\sigma _{n},\sigma _{n+1})$, in an
appropriately chosen coordinate system, $Y_{t}$ is given by Lemma \ref{lemma
for a half-space}. For $t<\xi $, $Y_{t}$ is given by the rotation $R\left(
\varphi _{t}\right) $ of $G\left( Y_{0}-X_{0}\right) X_{t}$ which leaves
invariant the first $\left( d-2\right) $ coordinates, and therefore%
\begin{equation*}
\angle \left( m_{t},\nu _{\mathcal{H}}\right) =\angle \left( m_{0},\nu _{%
\mathcal{H}}\right) +\frac{L_{t}^{0}-L_{t}^{\pi }}{2},
\end{equation*}%
which proves the claim in this case (note that before the coupling time $\xi
$ only one of the non-decreasing processes $L_{t}^{0}$ and $L_{t}^{\pi }$ is
not identically zero).

Since for $t\geq \xi $ we have $Y_{t}=\left( X_{t}^{1},\ldots ,\left\vert
X_{t}^{d}\right\vert \right) $, we have $\angle \left( m_{t},\nu _{\mathcal{H%
}}\right) =0$ which concludes the proof of the claim.

ii) On the time interval $[\sigma _{0},\sigma _{1})$ the processes are
coupled, so there is nothing to prove in this case.

On the time interval $[\sigma _{1},\sigma _{2})$, in an appropriately chosen
coordinate system we have $Y_{t}=\left( X_{t}^{1},\ldots ,\left\vert
X_{t}^{d}\right\vert \right) $, thus the mirror $M_{t}$ coincides with the
boundary hyperplane $\mathcal{H}_{d}=\left\{ \left( z^{1},\ldots
,z^{d}\right) \in \mathbb{R}^{d}:z^{d}=0\right\} $ of $D_{2}$ where the
reflection takes place, thus $M_{t}\cap D_{2}=\varnothing $ in this case.

Inductively, assume the claim is true for $t<\sigma _{n}$. By continuity, $%
M_{\sigma _{n}}\cap D_{2}=\varnothing $, thus $D_{2}$ lies on one side of $%
M_{\sigma _{n}}$. By the previous proof, the angle $\angle \left( m_{t},\nu
_{\mathcal{H}}\right) $ between $m_{t}$ and the inward unit normal $\nu _{%
\mathcal{H}}$ to bounding hyperplane $\mathcal{H}$ of $D_{2}$ where the
reflection takes place decreases to zero; since $D_{2}$ is a convex domain,
it follows that on the time interval $[\sigma _{n},\sigma _{n+1})$ we have $%
M_{t}\cap D_{2}=\varnothing $, concluding the proof.

iii) The first part of the claim follows from the previous proof (when the
processes are not coupled, the mirror (hence $X_{t}$) lies outside $D_{2}$;
by continuity, it follows that at the coupling time $\xi $ we must have $%
X_{\xi }=Y_{\xi }\in \partial D_{2}$).

To prove the second part of the claim, consider an arbitrary time interval $%
[\sigma _{n},\sigma _{n+1})$ between two successive hits of $Y_{t}$
to different bounding hyperplanes of $D_{2}$. In an appropriately
chosen coordinate system, $Y_{t}$ is given by Lemma \ref{lemma for a
half-space}. After the coupling time $\xi $, $Y_{t}$ is given by
$Y_{t}=\left( X_{t}^{1},\ldots ,\left\vert X_{t}^{d}\right\vert
\right) $, and therefore if $X_{t}\in D_{2}$ (thus $X_{t}^{d}\geq
0$) we have $Y_{t}=\left( X_{t}^{1},\ldots ,X_{t}^{d}\right)
=X_{t}$, concluding the proof.

iv) Let $M_{t}^{\alpha }$ and $M_{t}^{\beta }$ denote the mirrors of the
coupling in $D^{\alpha }$, respectively $D^{\beta }$, with the same driving
Brownian motion $X_{t}$.

Since $Y_{t}^{\alpha }$ and $X_{t}$ are symmetric with respect to $%
M_{t}^{\alpha }$, and $Y_{t}^{\beta }$ and $X_{t}$ are symmetric with
respect to $M_{t}^{\beta }$, it follows that $Y_{t}^{\beta }$ is obtained
from $Y_{t}^{\beta }$ by a rotation which leaves invariant the hyperplane $%
M_{t}^{\alpha }\cap M_{t}^{\beta }$, or by a translation by a vector
orthogonal to $M_{t}^{\alpha }$ (in the case when $M_{t}^{\alpha }$ and $%
M_{t}^{\beta }$ are parallel).

The angle of rotation (respectively the vector of translation) is altered
only when either $Y_{t}^{\alpha }$ or $Y_{t}^{\beta }$ are on the boundary
of $D_{\alpha }$, respectively $D_{\beta }$. Since $D_{\alpha }\subset
D_{\beta }$ are convex domains, the angle of rotation (respectively the
vector of translation) decreases when $Y_{t}^{\beta }\in D_{\beta }$ or when
$Y_{t}^{\alpha }\in \partial D_{\alpha }$ and $\left( Y_{t}^{\beta
}-Y_{t}^{\alpha }\right) \cdot \nu _{D_{\alpha }}\left( Y_{t}^{\alpha
}\right) >0$ (in these cases $Y_{t}^{\beta }$ and $Y_{t}^{\alpha }$ receive
a push such that the distance $\left\vert Y_{t}^{\alpha }-Y_{t}^{\beta
}\right\vert $ is decreased), thus the maximum distance $\left\vert
Y_{t}^{\alpha }-Y_{t}^{\beta }\right\vert $ is attained when $Y_{t}^{\alpha
}\in \partial D_{\alpha }$ and $\left( Y_{t}^{\beta }-Y_{t}^{\alpha }\right)
\cdot \nu _{D_{\alpha }}\left( Y_{t}^{\alpha }\right) \leq 0$, and the
formula follows.
\end{proof}

\section{The proof of Theorem \protect\ref{main theorem}\label{The proof of
the main theorem}}

By Remark \ref{reduction to the case D_1=R}, it suffices to consider the
case when $D_{1}=\mathbb{R}^{d}$ and $D_{2}\subset \mathbb{R}^{d}$ is a
convex bounded domain with smooth boundary. To simplify the notation, we
will drop the index and write $D$ for $D_{2}$ in the sequel.

Let $\left( D_{k}\right) _{k\geq 1}$ be an increasing sequence of convex
polygonal domains in $\mathbb{R}^{d}$ with $\overline{D_{n}}\subset D_{n+1}$
and $\cup _{n\geq 1}D_{n}=D$.

Consider $\left( Y_{t}^{n},X_{t}\right) _{t\geq 0}$ a sequence of mirror
couplings in $\left( D_{n},\mathbb{R}^{d}\right) $ with starting point $x\in
D_{1}$, with driving Brownian motion $\left( W_{t}\right) _{t\geq 0}$, $%
W_{0}=0$ given by Theorem \ref{coupling in polygonal domains}.

By Proposition \ref{properties of mirror coupling in polygonal domains}, for
any $t>0$ we have%
\begin{equation*}
\sup_{s\leq t}\left\vert Y_{s}^{m}-Y_{s}^{n}\right\vert \leq \limfunc{Dist}%
\left( D_{n},D_{m}\right) =\max_{\substack{ x_{n}\in \partial D_{n},x_{m}\in
\partial D_{m}  \\ \left( x_{m}-x_{n}\right) \cdot \nu _{D_{n}}\left(
x_{n}\right) \leq 0}}\left\vert x_{n}-x_{m}\right\vert \underset{%
n,m\rightarrow \infty }{\rightarrow }0,
\end{equation*}%
hence $Y_{t}^{n}$ converges a.s. in the uniform topology to a continuous
process $Y_{t}$.

Since $\left( Y^{n}\right) _{n\geq 1}$ are reflecting Brownian motions in $%
\left( D_{n}\right) _{n\geq 1}$ and $D_{n}\nearrow D$, the law of $Y_{t}$ is
that of a reflecting Brownian motion in $D$, that $Y_{t}$ is a reflecting
Brownian motion in $D$ starting at $x\in D$ (see \cite{Burdzy-Weak
convergence}). Also note that since $Y_{t}^{n}$ are adapted to the
filtration $\mathcal{F}^{W}=\left( \mathcal{F}_{t}\right) _{t\geq 0}$
generated by the Brownian motion $W_{t}$, so is $Y_{t}$.

By construction, the driving Brownian motion $Z_{t}^{n}$ of $Y_{t}^{n}$
satisfies%
\begin{equation*}
Z_{t}^{n}=\int_{0}^{t}G\left( Y_{t}^{n}-X_{t}\right) dW_{t},\qquad t\geq 0.
\end{equation*}

Consider the process%
\begin{equation*}
Z_{t}=\int_{0}^{t}G\left( Y_{t}-X_{t}\right) dB_{t},
\end{equation*}%
and note that since $Y$ is $\mathcal{F}^{W}$-adapted and $\left\vert
\left\vert G\right\vert \right\vert =1$, by L\'{e}vy's characterization of
Brownian motion, $Z_{t}$ is a free $d$-dimensional Brownian motion starting
at $Z_{0}=0$, also adapted to the filtration $\mathcal{F}^{W}$.

We will show that $Z$ is the driving process of the reflecting Brownian
motion $Y_{t}$, i.e. we have%
\begin{equation*}
Y_{t}=x+Z_{t}+L_{t}^{Y}=x+\int_{0}^{t}G\left( Y_{s}-B_{s}\right)
dW_{s}+L_{t}^{Y},\qquad t\geq 0.
\end{equation*}

Note that the mapping $z\longmapsto G\left( z\right) $ is continuous with
respect to the norm $\left\vert \left\vert A\right\vert \right\vert
=\left\vert \left\vert \left( a_{ij}\right) \right\vert \right\vert
=\sum_{i,j=1}^{d}a_{ij}^{2}$ of $d\times d$ matrices at all points $z\in
\mathbb{R}^{d}-\left\{ 0\right\} $, hence $G\left( Y_{s}^{n}-X_{s}\right)
\underset{n\rightarrow \infty }{\rightarrow }G\left( Y_{s}-X_{s}\right) $ if
$Y_{s}-X_{s}\neq 0$. If $Y_{s}-X_{s}=0$, then either $Y_{s}=B_{s}\in D$ or $%
Y_{s}=X_{s}\in \partial D$.

If $Y_{s}=B_{s}\in D$, since $D_{n}\nearrow D$, there exists $N\geq 1$ such
that $B_{s}\in D_{N}$, hence $B_{s}\in D_{n}$ for all $n\geq N$. By
Proposition \ref{properties of mirror coupling in polygonal domains}, it
follows that $Y_{s}^{n}=B_{s}$ for all $n\geq N$, hence in this case we also
have $G\left( Y_{s}^{n}-B_{s}\right) =G\left( 0\right) \underset{%
n\rightarrow \infty }{\rightarrow }G\left( 0\right) =G\left(
Y_{s}-B_{s}\right) $.

If $Y_{s}=B_{s}\in \partial D$, since $\overline{D_{n}}\subset D$ we have $%
Y_{s}^{n}-B_{s}\neq 0$, and therefore by the definition of $G$ we have:%
\begin{eqnarray*}
&&\int_{0}^{t}\left\vert \left\vert G\left( Y_{s}^{n}-X_{s}\right) -G\left(
Y_{s}-X_{s}\right) \right\vert \right\vert ^{2}1_{Y_{s}=B_{s}\in \partial
D}ds \\
&=&\int_{0}^{t}\left\vert \left\vert H\left( \frac{Y_{s}^{n}-X_{s}}{%
\left\vert \left\vert Y_{s}^{n}-X_{s}\right\vert \right\vert }\right)
-I\right\vert \right\vert ^{2}1_{Y_{s}=X_{s}\in \partial D}ds \\
&=&\int_{0}^{t}\left\vert \left\vert I-2\frac{Y_{s}^{n}-X_{s}}{\left\vert
\left\vert Y_{s}^{n}-X_{s}\right\vert \right\vert }\left( \frac{%
Y_{s}^{n}-X_{s}}{\left\vert \left\vert Y_{s}^{n}-X_{s}\right\vert
\right\vert }\right) ^{\prime }-I\right\vert \right\vert
^{2}1_{Y_{s}=X_{s}\in \partial D}ds \\
&=&\int_{0}^{t}\left\vert \left\vert 2\frac{Y_{s}^{n}-X_{s}}{\left\vert
\left\vert Y_{s}^{n}-X_{s}\right\vert \right\vert }\left( \frac{%
Y_{s}^{n}-X_{s}}{\left\vert \left\vert Y_{s}^{n}-X_{s}\right\vert
\right\vert }\right) ^{\prime }\right\vert \right\vert ^{2}1_{Y_{s}=X_{s}\in
\partial D}ds \\
&=&4\int_{0}^{t}1_{Y_{s}=X_{s}\in \partial D}ds \\
&\leq &4\int_{0}^{t}1_{\partial D}\left( Y_{s}\right) ds \\
&=&0,
\end{eqnarray*}%
since $Y_{t}$ is a reflecting Brownian motion in $D$, and therefore it
spends zero Lebesgue time on the boundary of $D$.

Since $\left\vert \left\vert G\right\vert \right\vert =1$, using the above
and the bounded convergence theorem we obtain%
\begin{equation*}
\lim_{n\rightarrow \infty }\int_{0}^{t}\left\vert \left\vert G\left(
Y_{s}^{n}-X_{s}\right) -G\left( Y_{s}-X_{s}\right) \right\vert \right\vert
^{2}ds=0,
\end{equation*}%
and therefore by Doob's inequality it follows that%
\begin{equation*}
E\sup_{s\leq t}\left\vert Z_{s}^{n}-Z_{s}\right\vert ^{2}\leq cE\left\vert
Z_{t}^{n}-Z_{t}\right\vert ^{2}\leq cE\int_{0}^{t}\left\vert \left\vert
G\left( Y_{s}^{n}-X_{s}\right) -G\left( Y_{s}-X_{s}\right) \right\vert
\right\vert ^{2}ds\underset{n\rightarrow \infty }{\rightarrow }0,
\end{equation*}%
for any $t\geq 0$, which shows that $Z_{t}^{n}$ converges uniformly on
compact sets to $Z_{t}=\int_{0}^{t}G\left( Y_{s}-X_{s}\right) dW_{s}$.

From the construction, $Z_{t}^{n}$ is the driving Brownian motion for $%
Y_{t}^{n}$, that is%
\begin{equation*}
Y_{t}^{n}=x+Z_{t}^{n}+\int_{0}^{t}\nu _{D_{n}}\left( Y_{s}^{n}\right)
dL_{s}^{Y_{n}},
\end{equation*}%
and passing to the limit with $n\rightarrow \infty $ we obtain%
\begin{equation*}
Y_{t}=x+Z_{t}+A_{t}=x+\int_{0}^{t}G\left( Y_{s}-X_{s}\right)
dW_{s}+A_{t},\qquad t\geq 0,
\end{equation*}%
where $A_{t}=\lim_{n\rightarrow \infty }\int_{0}^{t}\nu _{D_{n}}\left(
Y_{s}^{n}\right) dL_{s}^{Y_{n}}$.

It remains to show that $A_{t}$ is a process of bounded variation. For an
arbitrary partition $0=t_{0}<t_{1}<\ldots t_{l}=t$ of $\left[ 0,t\right] $
we have%
\begin{eqnarray*}
E\sum_{i=1}^{l}\left\vert A_{t_{i}}-A_{t_{i-1}}\right\vert
&=&\lim_{n\rightarrow \infty }E\sum_{i=1}^{l}\left\vert
\int_{t_{i-1}}^{t_{i}}\nu _{D_{n}}\left( Y_{s}^{n}\right)
dL_{s}^{Y_{n}}\right\vert \\
&\leq &\lim \sup E~L_{t}^{Y_{n}} \\
&=&\lim \sup \int_{0}^{t}\int_{\partial D_{n}}p_{D_{n}}\left( s,x,y\right)
\sigma _{n}\left( dy\right) ds \\
&\leq &c\sqrt{t},
\end{eqnarray*}%
where $\sigma _{n}$ is the surface measure on $\partial D_{n}$ and the last
inequality above follows from the estimates in \cite{Bass-Some potential
theory} on the Neumann heat kernels $p_{D_{n}}\left( t,x,y\right) $ (see the
remarks preceding Theorem 2.1 and the proof of Theorem 2.4 in \cite%
{Burdzy-Coalescence}).

From the above it follows that $A_{t}=Y_{t}-x-Z_{t}$ is a continuous, $%
\mathcal{F}^{W}$-adapted process ($Y_{t}$, $Z_{t}$ are continuous, $\mathcal{F%
}^{W}$-adapted processes) of bounded variation.

By the uniqueness in the Doob-Meyer semimartingale decomposition of $Y_{t}$
- reflecting Brownian motion in $D$, it follows that%
\begin{equation*}
A_{t}=\int_{0}^{t}\nu _{D}\left( Y_{s}\right) dL_{s}^{Y},\qquad t\geq 0,
\end{equation*}%
where $L^{Y}$ is the local time of $Y$ on the boundary $\partial D$, and
therefore the reflecting Brownian motion $Y_{t}$ in $D$ constructed above is
a strong solution to%
\begin{equation*}
Y_{t}=x+\int_{0}^{t}G\left( Y_{s}-X_{s}\right) dW_{s}+\int_{0}^{t}\nu
_{D}\left( Y_{s}\right) dL_{s}^{Y},\qquad t\geq 0,
\end{equation*}%
or equivalent, the driving Brownian motion $Z_{t}=\int_{0}^{t}G\left(
Y_{s}-X_{s}\right) dW_{s}$ of $Y_{t}$ is a strong solution to
\begin{equation*}
Z_{t}=\int_{0}^{t}G\left( \tilde{\Gamma}\left( y+Z\right) _{s}-X_{s}\right)
dW_{s},\qquad t\geq 0,
\end{equation*}%
concluding the proof of Theorem \ref{main theorem}.

\section{Extensions and applications\label{Extensions and applications}}

As an application of the construction of mirror coupling, we will present a
unifying proof of the two most important results on Chavel's conjecture.

It is not difficult to prove that the Dirichlet heat kernel is an increasing
function with respect to the domain. Since for the Neumann heat kernel $%
p_{D}\left( t,x,y\right) $ of a smooth bounded domain $D\subset \mathbb{R}%
^{d}$ we have%
\begin{equation*}
\lim_{t\rightarrow \infty }p_{D}\left( t,x,y\right) =\frac{1}{\limfunc{vol}%
\left( D\right) },
\end{equation*}%
the monotonicity in the case of the Neumann heat kernel should be reversed.

The above observation was conjectured by Isaac Chavel (\cite{Chavel}), as
follows:

\begin{conjecture}[Chavel's conjecture, \protect\cite{Chavel}]
\label{Chavel's conjecture}Let $D_{1,2}\subset \mathbb{R}^{d}$ be smooth
bounded convex domains in $\mathbb{R}^{d}$, $d\geq 1$, and let $%
p_{D_{1}}\left( t,x,y\right) $, $p_{D_{2}}\left( t,x,y\right) $ denote the
Neumann heat kernels in $D_{1}$, respectively $D_{2}$. If $D_{2}\subset D_{1}
$, then
\begin{equation}
p_{D_{1}}\left( t,x,y\right) \leq p_{D_{2}}\left( t,x,y\right) ,
\label{Chavel inequality}
\end{equation}%
for any $t\geq 0$ and $x,y\in D_{1}$.
\end{conjecture}

\begin{remark}
The smoothness assumption in the above conjecture is meant to insure the
a.e. existence the inward unit normal to the boundaries of $D_{1}$ and $%
D_{2} $, that is the boundary should have locally a differentiable
parametrization. Requiring that the boundary of the domain is of class $%
C^{1,\alpha }$ ($0<\alpha <1$) is a convenient hypothesis on the smoothness
of the domains $D_{1,2}$.

In order to simplify the proof, we will assume that $D_{1,2}$ are smooth $%
C^{2}$ domains (the proof can be extended to a more general setup, by
approximating $D_{1,2}$ by less smooth domains).
\end{remark}

Among the positive results on Chavel conjecture, the most general known
results (and perhaps the easiest to use in practice) are due to I. Chavel
and W. Kendall (see \cite{Chavel}, \cite{Kendall}), and they show that if
there exists a ball $B$ centered at either $x$ or $y$ such that $%
D_{2}\subset B\subset D_{1}$, then the inequality (\ref{Chavel inequality})
in Chavel's conjecture holds true for any $t>0$.

While there are also other positive results which suggest that Chavel's
conjecture is true (see for example \cite{CarmonaZheng}, \cite{Hsu}), in
\cite{Bass-Burdzy} R. Bass and K. Burdzy showed that Chavel's conjecture
does not hold in its full generality (i.e. without additional hypotheses).

We note that both the proof of Chavel (the case when $D_{1}$ is a
ball centered at either $x$ or $y$) and Kendall (the case when
$D_{2}$ is a ball centered at either $x$ or $y$) relies in an
essential way that one of the domains is a ball: the first uses an
integration by parts technique, while the later uses a coupling
argument of the radial parts of Brownian motion, and none of them
can be applied to the other case.

Using the mirror coupling, we can derive a simple, unifying proof of these
two important results, as follows:

\begin{theorem}
\label{Theorem on Chavel's conjecture}Let $D_{2}\subset D_{1}\subset \mathbb{%
R}^{d}$ be smooth bounded domains and assume that $D_{2}$ is convex. If for $%
x,y\in D_{2}$ there exists a ball $B$ centered at either $x$ or $y$ such
that $D_{2}\subset B\subset D_{1}$, then for all $t\geq 0$ we have%
\begin{equation}
p_{D_{1}}\left( t,x,y\right) \leq p_{D_{2}}\left( t,x,y\right) .
\label{Chavel's inequality}
\end{equation}
\end{theorem}

\begin{proof}
Consider $x,y\in D_{2}$ fixed and assume without loss of generality that $%
D_{2}\subset B=B\left( y,R\right) \subset D_{1}$ for some $R>0$.

Consider a mirror coupling $\left( X_{t},Y_{t}\right) $ of reflecting
Brownian motions in $\left( D_{1},D_{2}\right) $ starting at $y\in D_{2}$.

The idea of the proof is to show that at all times $Y_{t}$ is at a distance
from $y$ smaller than (or equal) to that of $X_{t}$ from $y$.

To prove the claim, consider a time $t_{0}\geq 0$ when the processes are at
the same distance from $y$, that is $\left\vert Y_{t_{0}}-y\right\vert
=\left\vert X_{t_{0}}-x\right\vert $. If $X_{t_{0}}=Y_{t_{0}}$, for $t\geq
t_{0}$ the distances from $X_{t}$ and $Y_{t}$ to $y$ will remain equal until
the time $t_{1}$ when the processes hit the boundary of $D_{2}$, and $Y_{t}$
receives a push in the direction of the inward unit normal to the boundary
of $D_{2}$. Since $D_{2}$ is convex, this decreases the distance of $Y_{t}$
from $y$, and the claim follows in this case.

If the processes are decoupled and $\left\vert Y_{t_{0}}-y\right\vert
=\left\vert X_{t_{0}}-x\right\vert $, the hyperplane $M_{t_{0}}$ of symmetry
between $X_{t_{0}}$ and $Y_{t_{0}}$ passes through $y$, and the ball
condition shows that we cannot have $X_{t_{0}}\in \partial D_{1}$. Therefore
for $t\geq t_{0}$, the processes $X_{t}$ and $Y_{t}$ will remain at the same
distance from $y$ until $Y_{t}\in \partial D_{2}$, when the distance of $%
Y_{t}$ from $y$ is again decreased by the local push received as in the
previous case, concluding the proof of the claim.

Therefore, for any $\varepsilon >0$ we have%
\begin{equation*}
P^{x}\left( \left\vert X_{t}-y\right\vert <\varepsilon \right) \leq
P^{x}\left( \left\vert Y_{t}-y\right\vert <\varepsilon \right) ,
\end{equation*}%
and dividing by the volume of the ball $B\left( y,\varepsilon \right) $ and
passing to the limit with $\varepsilon \searrow 0$, from the continuity of
the transition density of the reflecting Brownian motion in the space
variable we obtain
\begin{equation*}
p_{D_{1}}\left( t,x,y\right) \leq p_{D_{2}}\left( t,x,y\right) ,
\end{equation*}%
for any $t\geq 0$, concluding the proof of the theorem.
\end{proof}

\begin{remark}
As also pointed out by Kendall in \cite{Kendall} (the case when $D_{2}$ is a
ball), we note that the convexity of the larger domain $D_{1}$ is not needed
in the above proof in order to derive the validity of condition (\ref{Chavel
inequality}) in Chavel's conjecture.
\end{remark}

\begin{remark}
\label{possible extensions of Chavel}We also note that the above proof uses
only geometric considerations on the relative position of the reflecting
Brownian motions coupled by mirror coupling. Analytically, the above proof
reduces to showing that $R_{t}=\left\vert X_{\alpha _{t}}\right\vert
^{2}-\left\vert Y_{\alpha _{t}}\right\vert ^{2}\geq 0$ for all $0\leq t<\xi
=\inf \left\{ s>0:X_{s}=Y_{s}\right\} $, where $R_{t}$ is the solution of
the following stochastic differential equation%
\begin{equation}
R_{t}=R_{0}+2\int_{0}^{t}R_{s}dB_{s}+2S_{t},
\end{equation}%
where $B_{t}=\int_{0}^{\alpha _{t}}\frac{X_{s}-Y_{s}}{\left\vert
X_{s}-Y_{s}\right\vert ^{2}}\cdot dW_{s}$ is $1$-dimensional Brownian
motion, $\alpha _{t}=A_{t}^{-1}$ is the inverse of the non-decreasing
process $A_{t}$ defined by%
\begin{equation*}
A_{t}=\int_{0}^{t}\frac{1}{\left\vert X_{s}-Y_{s}\right\vert ^{2}}ds,
\end{equation*}%
and
\begin{equation*}
S_{t}=\int_{0}^{t}X_{\alpha _{s}}\cdot \nu _{D_{1}}\left( X_{\alpha
_{s}}\right) dL_{\alpha _{s}}^{X}-\int_{0}^{t}Y_{\alpha _{s}}\cdot \nu
_{D_{2}}\left( Y_{\alpha _{s}}\right) dL_{\alpha _{s}}^{Y}.
\end{equation*}

Perhaps a better understanding of the mirror coupling, based on the analysis
of the local times $L^{X}$ and $L^{Y}$ spent by $X_{t}$ and $Y_{t}$ on the
boundaries of $D_{1}$, respectively $D_{2}$, in connection to the geometry
of the boundaries $\partial D_{1}$ and $\partial D_{2}$ could give a proof
of Chavel's conjecture for some new classes of convex domains, but so far we
were unable to implement it.
\end{remark}

We have chosen to carry out the construction of the mirror coupling
in the case of smooth domains with $\overline{D_{2}}\subset D_{1}$
and $D_{2}$ convex, having in mind the application to Chavel's
conjecture. However, although the technical details can be
considerably longer, it is possible to construct the mirror coupling
in a more general setup.

For example, in the case when $D_{1}$ and $D_{2}$ are disjoint domains, none
of the difficulties encountered in the construction of the mirror coupling
occur (the possibility of coupling/decoupling), so the constructions extends
immediately to this case.

The two key ingredients in our construction of the mirror coupling were the
hypothesis $\overline{D_{2}}\subset D_{1}$ (needed in order to reduce by a
localization argument the construction to the case $D_{1}=\mathbb{R}^{d})$
and the hypothesis on the convexity of the inner domain $D_{2}$ (which
allowed us to construct a solution of the equation of the mirror coupling in
the case $D_{1}=\mathbb{R}^{d}$).

Replacing the first hypothesis by the condition that the boundaries $%
\partial D_{1}$ and $\partial D_{2}$ are not tangential (needed for the
localization of the construction of the mirror coupling) and the second one
by condition that $D_{1}\cap D_{2}$ is a convex domain, the arguments in the
present construction can be modified in order to give rise to a mirror
coupling of reflecting Brownian motion in $\left( D_{1},D_{2}\right) $.


\begin{figure}[thb]
\begin{center}
\includegraphics[scale=0.6]{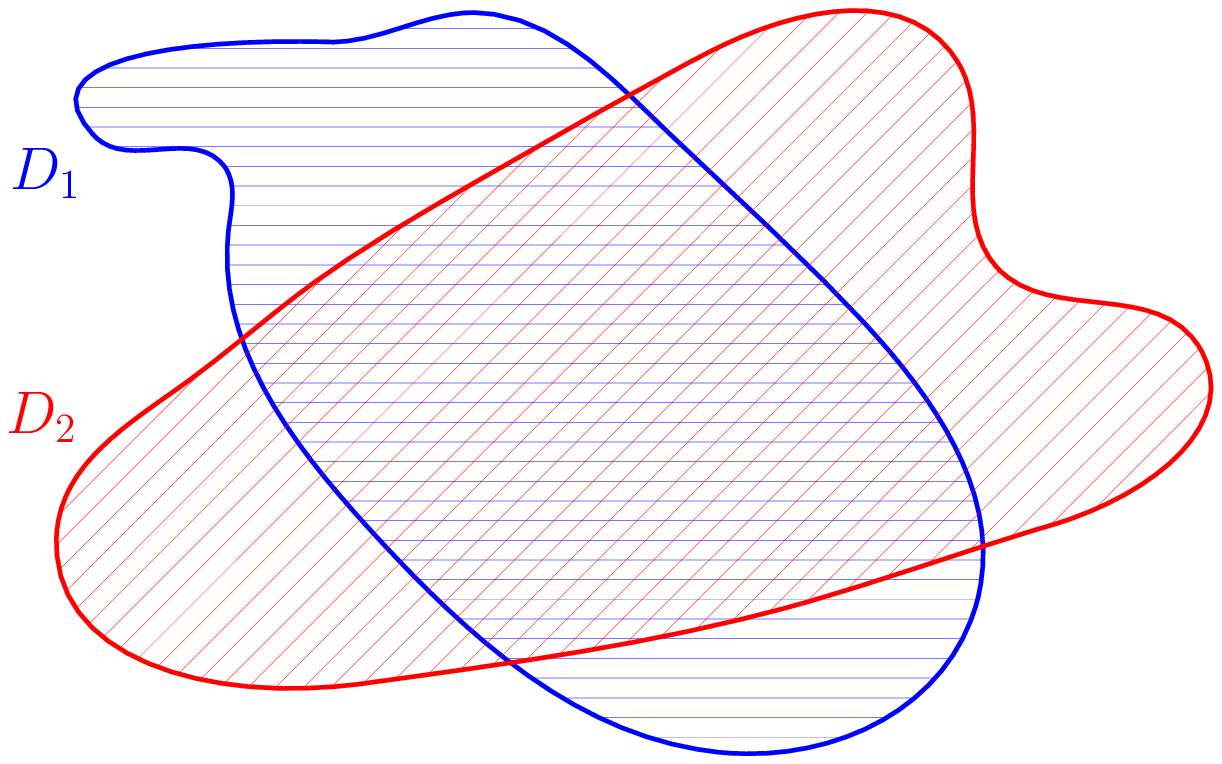}
\caption{Generic smooth domains $%
D_{1,2}\subset \mathbb{R}^{d}$ for the mirror coupling:
$D_{1},D_{2}$ have non-tangential boundaries and $D_{1}\cap D_{2}$
is a convex domain.} \label{fig3}
\end{center}
\end{figure}

We conclude with some remarks on the non-uniqueness of the mirror coupling
in general domains. To simplify the ideas, we will restrict to the $1$%
-dimensional case when $D_{2}=\left( 0,\infty \right) \subset D_{1}=\mathbb{R%
}$.

Fixing $x\in (0,\infty )$ as starting point of the mirror coupling $\left(
X_{t},Y_{t}\right) $ in $\left( D_{1},D_{2}\right) $, the equations of the
mirror coupling are
\begin{eqnarray}
X_{t} &=&x+W_{t} \\
Y_{t} &=&x+Z_{t}+L_{t}^{Y} \\
Z_{t} &=&\int_{0}^{t}G\left( Y_{s}-X_{s}\right) dW_{s}
\end{eqnarray}%
where in this case%
\begin{equation*}
G\left( z\right) =\left\{
\begin{tabular}{ll}
$-1,\qquad $ & if $z\neq 0$ \\
$+1,$ & if $z=0$%
\end{tabular}%
\right. .
\end{equation*}

Until the hitting time $\tau =\left\{ s>0:Y_{s}\in \partial
D_{2}\right\} $ of the boundary of $\partial D_{2}$ we have
$L_{t}^{Y}\equiv 0$, and with the substitution
$U_{t}=-\frac{1}{2}\left( Y_{t}-X_{t}\right) $, the
stochastic differential for $Y_{t}$ becomes%
\begin{equation}
U_{t}=\int_{0}^{t}\frac{1-G\left( Y_{s}-X_{s}\right) }{2}dW_{s}=\int_{0}^{t}%
\sigma \left( U_{s}\right) dW_{s},  \label{degenerate SDE}
\end{equation}%
where%
\begin{equation*}
\sigma \left( z\right) =\frac{1-G\left( z\right) }{2}=\left\{
\begin{tabular}{ll}
$1,\qquad $ & if $z\neq 0$ \\
$0,$ & if $z=0$%
\end{tabular}%
\right. .
\end{equation*}

By a result of Engelbert and Schmidt (\cite{Engelbert-Schmidt}) the solution
of the above problem is not even weakly unique, for in this case the set of
zeroes of the function $\sigma $ is $N=\{0\}$ and $\sigma ^{-2}$ is locally
integrable on $\mathbb{R}$.

In fact, more can be said about the solutions of (\ref{degenerate SDE}) in
this case. It is immediate that both $U_{t}\equiv 0$ and $U_{t}=W_{t}$ are
solutions to \ref{degenerate SDE}, and it can be shown that an arbitrary
solution can be obtained from $W_{t}$ by delaying it when it reaches the
origin (sticky Brownian motion with sticky point the origin).

Therefore, until the hitting time $\tau $ of the boundary, we obtain
as solutions
\begin{equation}
Y_{t}=X_{t}=x+W_{t}  \label{sol1}
\end{equation}%
and
\begin{equation}
Y_{t}=X_{t}-2W_{t}=x-W_{t},  \label{sol2}
\end{equation}%
and an intermediate range of solutions, which agree with (\ref{sol1}) for
some time, then switch to (\ref{sol2}) (see \cite{Pascu}).

Correspondingly, this gives rise to mirror couplings of reflecting Brownian
motions for which the solutions stick to each other after they have coupled
(as in (\ref{sol1})), or they immediately split apart after coupling (as in (%
\ref{sol2})), and there is a whole range of intermediate
possibilities. The first case can be referred to as \emph{sticky
}mirror coupling, the second as \emph{non-sticky} mirror coupling,
and the intermediate possibilities as \emph{weak/mild sticky} mirror
coupling.

The same situation occurs in the general setup in $\mathbb{R}^{d}$,
and it is the cause of lack uniqueness of the stochastic
differential equations which define the mirror coupling. In the
present paper we detailed the construction of the sticky mirror
coupling, which we considered to be the most interesting, both from
the point of view of constructions and of the applications, although
the other types of mirror coupling might prove useful in other
applications.

\section*{Acknowledgement} I would like to thank Krzysztof Burdzy and
Wilfried S. Kendall for the helpful discussions and the
encouragement to undertake the task of the present project.

\end{document}